\documentclass[11pt,reqno]{amsart}
\usepackage{tikz-cd}
\usepackage{amssymb,amsmath}
\usepackage{array}

\usepackage{amsmath, amssymb, amsfonts, amstext, amsthm, amscd}
\usepackage[mathscr]{euscript}
\usepackage{enumerate}

\usepackage{float}
\usepackage{graphicx}

\usepackage[raiselinks=false,colorlinks=true,citecolor=blue,urlcolor=blue,linkcolor=blue,bookmarksopen=true,pdftex]{hyperref}

\usepackage[paperheight=11.2in,paperwidth=8.1in,bindingoffset=0in,left=0.9in,right=0.9in,top=1in,bottom=1in,headsep=.5\baselineskip]{geometry}
\usepackage{tikz}
\usepackage{tikz-cd}
\usetikzlibrary{snakes}
\usetikzlibrary{intersections, calc}

\usepackage[english]{babel}
\usepackage{chngcntr}

\begin{document}

\thispagestyle{empty}

\def\theequation{\arabic{section}.\arabic{equation}}

\newcommand{\codim}{\mbox{{\rm codim}$\,$}}
\newcommand{\stab}{\mbox{{\rm stab}$\,$}}
\newcommand{\lr}{\mbox{$\longrightarrow$}}

\newcommand{\be}{\begin{equation}}
\newcommand{\ee}{\end{equation}}

\newtheorem{guess}{Theorem}[section]
\newcommand{\bth}{\begin{guess}$\!\!\!${\bf }~}
\newcommand{\eeth}{\end{guess}}
\renewcommand{\bar}{\overline}
\newtheorem{propo}[guess]{Proposition}
\newcommand{\bpropo}{\begin{propo}$\!\!\!${\bf }~}
\newcommand{\epropo}{\end{propo}}

\newtheorem{lema}[guess]{Lemma}
\newcommand{\blem}{\begin{lema}$\!\!\!${\bf }~}
\newcommand{\elem}{\end{lema}}

\newtheorem{defe}[guess]{Definition}
\newcommand{\bdefe}{\begin{defe}$\!\!\!${\bf }~}
\newcommand{\edefe}{\end{defe}}

\newtheorem{coro}[guess]{Corollary}
\newcommand{\bcor}{\begin{coro}$\!\!\!${\bf }~}
\newcommand{\ecor}{\end{coro}}

\newtheorem{rema}[guess]{Remark}
\newcommand{\brem}{\begin{rema}$\!\!\!${\bf }~\rm}
\newcommand{\erem}{\end{rema}}

\newtheorem{exam}[guess]{Example}
\newcommand{\beg}{\begin{exam}$\!\!\!${\bf }~\rm}
\newcommand{\eeg}{\end{exam}}

\newtheorem{notn}[guess]{Notation}
\newcommand{\bnot}{\begin{notn}$\!\!\!${\bf }~\rm}
\newcommand{\enot}{\end{notn}}
\newcommand{\cw} {{\mathcal W}}
\newcommand{\ct} {{\mathcal T}}
\newcommand{\ch}{{\mathcal H}}
\newcommand{\cf}{{\mathcal F}}
\newcommand{\cd}{{\mathcal D}}
\newcommand{\cR}{{\mathcal R}}
\newcommand{\cv}{{\mathcal V}}
\newcommand{\cn}{{\mathcal N}}
\newcommand{\lra}{\longrightarrow}
\newcommand{\ra}{\rightarrow}
\newcommand{\blr}{\Big \longrightarrow}
\newcommand{\da}{\Big \downarrow}
\newcommand{\ua}{\Big \uparrow}
\newcommand{\hra}{\mbox{{$\hookrightarrow$}}}
\newcommand{\rt}{\mbox{\Large{$\rightarrowtail$}}}
\newcommand{\dua}{\begin{array}[t]{c}
\Big\uparrow \\ [-4mm]
\scriptscriptstyle \wedge \end{array}}
\newcommand{\ctext}[1]{\makebox(0,0){#1}}
\setlength{\unitlength}{0.1mm}
\newcommand{\cm}{{\mathcal M}}
\newcommand{\cl}{{\mathcal L}}
\newcommand{\cp}{{\mathcal P}}
\newcommand{\ci}{{\mathcal I}}
\newcommand{\bz}{\mathbb{Z}}
\newcommand{\cs}{{\mathcal s}}
\newcommand{\ce}{{\mathcal E}}
\newcommand{\ck}{{\mathcal K}}
\newcommand{\cz}{{\mathcal Z}}
\newcommand{\cg}{{\mathcal G}}
\newcommand{\cj}{{\mathcal J}}
\newcommand{\cc}{{\mathcal C}}
\newcommand{\ca}{{\mathcal A}}
\newcommand{\cb}{{\mathcal B}}
\newcommand{\cx}{{\mathcal X}}
\newcommand{\co}{{\mathcal O}}
\newcommand{\bq}{\mathbb{Q}}
\newcommand{\bt}{\mathbb{T}}
\newcommand{\bh}{\mathbb{H}}
\newcommand{\br}{\mathbb{R}}
\newcommand{\bl}{\mathbf{L}}
\newcommand{\wt}{\widetilde}
\newcommand{\im}{{\rm Im}\,}
\newcommand{\bc}{\mathbb{C}}
\newcommand{\bp}{\mathbb{P}}
\newcommand{\ba}{\mathbb{A}}
\newcommand{\spin}{{\rm Spin}\,}
\newcommand{\ds}{\displaystyle}
\newcommand{\tor}{{\rm Tor}\,}
\newcommand{\bff}{{\bf F}}
\newcommand{\bs}{\mathbb{S}}
\def\ns{\mathop{\lr}}
\def\nssup{\mathop{\lr\,sup}}
\def\nsinf{\mathop{\lr\,inf}}
\renewcommand{\phi}{\varphi}
\newcommand{\tT}{{\widetilde{T}}}
\newcommand{\tG}{{\widetilde{G}}}
\newcommand{\tB}{{\widetilde{B}}}
\newcommand{\tC}{{\widetilde{C}}}
\newcommand{\tW}{{\widetilde{W}}}
\newcommand{\tphi}{{\widetilde{\Phi}}}
\newcommand{\fp}{{\mathfrak{p}}}

\title[Even-dimensional complex quadrics]{Equivariant $K$-theory of even-dimensional complex quadrics }

\author[B. Paul]{Bidhan Paul}
\address{Department of Mathematics, Indian Institute of Technology, Madras, Chennai 600036, India}
\email{ bidhanam95@gmail.com; bidhan@smail.iitm.ac.in}



\begin{abstract}
  The aim of this paper is to describe the torus equivariant
  $K$-ring of even-dimensional complex quadrics by studying the graph equivariant $K$-theory of their corresponding GKM graphs.  This involves providing a  presentation for its graph equivariant $K$- ring in terms of generators and relations. This parallels the description of the equivariant cohomology ring  of even-dimensional complex quadrics due to Kuroki.
\end{abstract}

\maketitle

\section{Introduction}

A GKM manifold is an \emph{equivariantly formal} manifold $M^{2k}$ endowed with an action of a compact torus $T := (S^1)^n$. This action satisfies the condition that the set of $0$ and $1$-dimensional orbits have the structure of a graph. This class of objects was initially introduced by Goresky, Kottwitz, and MacPherson in \cite{gkm} as a class of algebraic varieties (GKM derives from their initials). 

Inspired by their work, Guillemin and Zara \cite{gz01} introduced the concept of an \emph{abstract GKM graph} as an abstract graph with edges labeled by vectors in the dual of Lie algebra of $T$. In particular, an abstract GKM graph is a combinatorial counterpart of GKM manifolds and helps to study various topological and geometric properties of GKM manifolds (see e.g. \cite{gz01}, \cite{ghz06}, \cite{gkz}, \cite{ku16}) using combinatorial properties of GKM graph. On the other hand, abstract GKM graphs themselves have gathered attention beyond their geometric motivations (see e.g. \cite{fy19, fim14, k19, ku21, mmp07, y21}). These concepts gave rise to a rich field in mathematics commonly referred to as \emph{GKM theory}. This theory has been explored from both geometric and combinatorial perspectives and has found applications in various areas, such as representation theory (e.g \cite{fiebig}).

In \cite{gkm}, Goresky, Kottwitz, and MacPherson showed  that the equivariant cohomology ring $H_T(M)$ of a GKM pair $(M,~T)$ can be computed by studying the graph equivariant cohomology of the associated GKM graph $(\Gamma,\alpha)$. Later,  Knutson and Rosu proved  the much harder part \cite[Theorem A.4, Corollary A.5]{kr}  that the results remain true for the equivariant $K$- theory ring $K_T(M)$ as well (see, \eqref{gkmkth}). It is noteworthy that although Knutson and Rosu primarily focused on equivariant $K$-theory with complex coefficients in \cite{kr}, the validity of \eqref{gkmkth} persists with integer coefficients due to the works of Vezzosi and Vistoli \cite[{Corollary 5.11}]{vv}, which is an algebraic analogue of \cite[Theorem A.4]{kr} (this fact is also remarked in \cite[p.447]{kr}).

A complex quadric in dimension $2n$ is defined to be the zero locus of a non-degenerate quadratic form in $2n+2$ complex variables and  can be represented as 
\[Q_{2n}:=\Big\{[x_1:\cdots:x_{2n+2}]\in\mathbb{CP}^{2n+1}~|~\sum_{j=1}^{n+1}x_ix_{2n+3-i}=0\Big\}.\] 
There exists the  natural $T^{n+1}$- action on $Q_{2n}$
\begin{equation}\label{introtorusaction}
	[x_1:\cdots:x_{2n+2}]\cdot(t_1,\ldots,t_{n+1}):=[x_1t_1:\cdots:x_{n+1}t_{n+1}:t_{n+1}^{-1}x_{n+2}:\cdots:t_1^{-1}x_{2n+2}].
\end{equation}
where $(t_1,\ldots,t_{n+1})\in T^{n+1}$. Since, $Q_{2n}$ is diffeomorphic with the Grassmannian $\br_{2,2n}\cong SO(2n+2)/(SO(2n)\times SO(2))$ of oriented 2-planes through the origin in $\br^{2n+2}$, the above action is equivalent to a maximal torus $T^{n+1}$- action  which is induced by restricting the transitive $SO(2n+2)$- action. Further, since $SO(2n)\times SO(2)\subset SO(2n+2)$ is a subgroup of maximal rank (i.e $T^{n+1}$ is also a maximal torus of $SO(2n+2)$), the set of $0$ and $1$-dimensional orbits of $T^{n+1}$- action have the structure of a graph (see, \cite{ghz06}). Therefore, the GKM graph of $Q_{2n}$ with $T^{n+1}$- action \eqref{introtorusaction} can be determined by labeling the edges with tangential representations. Note that, the torus action in \eqref{introtorusaction} is not effective as it has a non-trivial kernel $\mathbb{Z}_2=\{\pm1\}$. Alternatively, one can consider the $T^{n+1}/\mathbb{Z}_2~(\cong T^{n+1})$-action on $Q_{2n}$, which is effective (see, Section \ref{edcqigg}). 

In \cite{lai72, lai74}, H. Lai  studied the  ordinary  cohomology ring (over $\bz$) and  homotopy groups  of even-dimensional complex quadrics $Q_{2n}$  explicitly. Since $H^{odd}(Q_{2n})=0$ by \cite[Theorem 3.2]{lai74}, $Q_{2n}$ is an equivariantly formal GKM manifold (see, \cite{gkm}). Therefore, the equivariant $K$-theory $K_{T^{n+1}}(Q_{2n})$ (resp. equivariant cohomology $H_{T^{n+1}}(Q_{2n})$) of the effective $T^{n+1}$ -action can be determined from the graph equivariant $K$- theory (resp. graph equivariant cohomology) of its GKM graph .

Recently, in \cite{ku23}, S. Kuroki   studied the GKM graph  of $Q_{2n}$ under the effective $T^{n+1}$- action and  computed the graph equivariant cohomology of their GKM graphs.  In particular, the author has provided a ring structure for the equivariant cohomology, presenting it in the form of generators and relations.

Our main aim in this paper is to determine the equivariant $K$- theory of the effective $T^{n+1}$- action on $Q_{2n}$ by explicitly  describing its generators and relations  in terms of the GKM graphs. This parallels the description of the equivariant cohomology ring  of even-dimensional complex quadrics due to Kuroki \cite{ku23}.

{\bf This article is structured as follows :}
In Section \ref{preli}, we provide an overview of fundamental concepts related to GKM manifolds, integral GKM graphs, and their $K$-ring. Moving on to Section \ref{edcqigg}, we define even-dimensional complex quadrics $Q_{2n}$ and present the associated GKM graph $\mathscr{Q}_{2n}$ arising from the effective $T^{n+1}$-action.

In Section \ref{grkring}, we introduce the generators $M_v$ and $\Delta_P$ of $K(\mathscr{Q}_{2n})$ and establish four types of relations among them. Additionally, we explore various properties of $M_v$ and $\Delta_P$ in this section. Subsequently, in Section \ref{secmainth}, we provide a proof for the main theorem of this paper (Theorem \ref{mainth}).

\section{Preliminaries}\label{preli}
In this section, we review some basic facts of GKM manifolds (see, \cite{k09}), integral GKM graphs (see, \cite{gz01}), and the $K$-ring of integral GKM graphs (see, \cite{gsz13}).

Let $M$ be a compact manifold of real dimension $2k$. Let $T$ be a compact $n$-dimensional torus, $\mathfrak{t}$ be its Lie algebra ($\mathfrak{t}^*$ be the dual of $\mathfrak{t}$), which acts effectively on $M$. 	A pair $(M,T)$ 	is said to be a \emph{GKM manifold} if the following conditions are satisfied :
	\begin{itemize}
		\item the fixed point set $M^T$ is finite.
		\item $M$ possess a $T$- invariant almost complex structure.
		\item for each element $p\in M^T$, the weights $\{\alpha_{i,p}\in\mathfrak{t}^*\,:\, 1\leq i\leq k\}$ of the isotropy representation $T_pM$ of $T$ are pairwise linearly independent.
	\end{itemize}

\subsection{An integral GKM graph and its $K$- ring}
Let $\Gamma=(\cv,~\ce)$ be a  regular graph of degree $d$, where $\cv$ and $\ce$ are the set of vertices and the set of oriented edges respectively.
For an oriented edge $e\in\mathcal{E}$, we define $i(e)$ (resp,  $t(e)$) as the {\it initial} (resp, {\it terminal})  vertex of $e$. Moreover, we let $\overline{e}$ to be the edge with  direction opposite to $e$. Let $\ce_p\subset \ce$ be the set of edges whose initial vertex is $p$.

\begin{defe}[Connection]\label{connection}
	Let $e =(pq)\in\ce$ be an oriented edge. A bijection $\nabla_e: \ce_p\to \ce_q$ with ${\nabla_e}(e)=\overline{e}$ is said to be a	connection along $e$. Furthermore, we define a \emph{connection} on graph $\Gamma$ to be a collection $\nabla:=\{\nabla_e\}_{e\in \ce}$ satisfying $\nabla_e^{-1}=\nabla_{\overline{e}}$ for every $e\in \ce$.
\end{defe}
Let $(\mathfrak{t}_{\bz})^*$ be the weight lattice of torus $T$.
\begin{defe}[Integral axial function]\label{intaxialfn}
	Let $\nabla$ be a connection on graph $\Gamma$.  A map $\alpha: \ce\to (\mathfrak{t}_\bz)^*$  is called a \emph{$\nabla$-compatible integral axial function} on $\Gamma$ if it satisfies the following conditions :
	\begin{enumerate}
		\item $\alpha(\overline{e})=-\alpha(e)$ for all $e\in\ce$.
		\item the vectors $\{\alpha(e)\,:\,e\in \ce_p\}$ are pairwise linearly independent for all $p\in \cv$.
		\item for every edge $e=(pq)$ and every $e'\in \ce_p$, there exists $m(e,e')\in\bz$ depending on $e,\, e'$ such that \[\alpha(\nabla_e(e'))-\alpha(e')=m(e,e')\alpha(e).\]
	\end{enumerate}
	An \emph{integral axial function} on $\Gamma$ is a  $\nabla$-compatible integral axial function for some connection $\nabla$ on $\Gamma$.
\end{defe}
\begin{defe}[Integral GKM graph]
	An integral GKM graph is a pair $(\Gamma,\alpha,\nabla)$ or $(\Gamma,\alpha)$ (if $\nabla$ is obviously determined) which consists of a graph $\Gamma=(\cv,\ce)$ and an integral axial function $\alpha:\ce\to(\mathfrak{t}_\bz)^*$.
\end{defe}
\begin{rema}\label{uniqueconn}
For an integral GKM graph $(\Gamma, \alpha, \nabla)$, if every three-tuple of elements in $\alpha(\ce_p)$ is linearly independent for all vertices $p\in\cv$, then $\nabla$ is uniquely determined {(see, \cite[p.453]{mmp07}, \cite[Proposition 2.1.3]{gz01})}. 
\end{rema}
 Let  $R(T)$ be the representation ring of $T$ which can be identified with the ring of finite sums
\[\sum_{k=1}^lm_ke^{2\pi\sqrt{-1}\alpha_k}\] for $m_k\in\bz$ and $\alpha_k\in(\mathfrak{t}_\bz)^*$. Let $Maps(\cv,\, R(T))$ be the ring of maps from the set of vertices to the representation ring of $T$. 

\begin{defe}[$K$- ring of an integral GKM graph]\label{kclass}
	 An element $f\in  Maps(\cv,\, R(T))$ is said to be a {\it $K$- class} if for every edge $e=(x y)\in \ce$, 
	\begin{equation}\label{kclasscond}
		f(x)-f(y)=\beta(1-e^{2\pi\sqrt{-1}\alpha(e)})
	\end{equation}
	for some $\beta\in R(T)$. 
	Observe that the collection of all $K$-classes forms a subring of $Maps(\cv,\, R(T))$, known as the {\bf $K$- ring} of $(\Gamma,\alpha)$ and denoted by $K_{\alpha}(\Gamma)$ (see, \cite{gsz13} for further details).
\end{defe}
Note that $K_{\alpha}(\Gamma)$ has a $R(T)$- module structure which is induced by the  injective homomorphism $i: R(T)\to K_{\alpha}(\Gamma)$, where the image of $x\in R(T)$, denoted as $i(x)$, is defined by
\begin{equation}\label{algebramap}
	i(x)(v)=x ~~~~\text{for all $v\in\cv$.}\end{equation}

\subsection{$K$-ring of GKM graphs induced from GKM manifolds}

	For a given GKM manifold $(M,T)$, one can define a graph $\Gamma_M :=(\cv, \ce)$  with vertices $\cv=M^T$. For $p,\,q\in M^T$, there is an oriented edge $e= (pq)\in \ce$ if and only if there exists a $T$-invariant embedded 2-sphere $S^2_e$ containing $p$ and $q$ as north and south poles respectively. The edges are labeled by a function $\alpha_M:\ce\to (\mathfrak{t}_\bz)^*$, known as \emph{axial function}, which takes an edge $e= (pq)\in\ce$  to the weight of the isotropy representation  $T_pS^2_e$ of $T$. 
	
	Further, since $(M,T)$ is a GKM manifold, a \emph{connection}  $\nabla$ can be canonically defined (see, Remark \ref{uniqueconn}) such that the above $\alpha_M:\ce\to (\mathfrak{t}_\bz)^*$ is an integral axial function. In particular, the pair $(\Gamma_M,\alpha_M)$ is an integral GKM graph associated to the GKM manifold $(M,T)$ (see, \cite[Section 1.1]{gz01} for further details).

\begin{rema}
Knutson and Rosu \cite[Theorem A.4, Colollary A.5]{kr} proved that, for a GKM graph $(\Gamma_M,\alpha_M)$ associated to the GKM pair  $(M,T)$, the following isomorphisim holds 
	\begin{equation}\label{gkmkth}
		K_T(M)\simeq K_{\alpha_M}(\Gamma_M),
	\end{equation} 
	where $K_T(M)$ is the $T$- equivariant $K$- ring of $M$.
	
Knutson and Rosu have studied equivariant $K$- theory with complex coeffecients in \cite{kr}. However, \eqref{gkmkth} remains true over $\bz$ as well due to the works of Vezzosi and Vistoli	\cite[Corollary 5.11]{vv}, which is an algebraic analogue of \cite[Theorem A.4]{kr} (this fact is also remarked in \cite[p.447]{kr}). 
\end{rema}

Here, we provide an overview of why the GKM pair $(M,T)$ fits into the algebro-geometric framework and satisfies the conditions of \cite[Corollary 5.11]{vv} (i.e., why the $T$-action admits sufficiently many limits). This reasoning shows that \cite[Corollary 5.11]{vv} can be applied to deduce the isomorphism between $K_T(M)$ and the $K$-theory ring of the GKM graph associated with $(M, T)$, with integer coefficients.

\begin{rema}
	For $1\leq i\leq n$, consider a 1-parameter subgroup $\lambda_i : \mathbb{G}_m =\bc\setminus\{0\} \to T$ defined by
\[t\mapsto \mbox{diag}(1,\ldots,1,t,1\ldots,1),\]
	where $t$ appears in the $(i,i)$-th position. Note that the closed points of $M$ are precisely the $T$-fixed points of $M$. Therefore, each $\lambda_i$ admits limits (see \cite[Definition 5.8]{vv}). Further, since $\{\lambda_i : i = 1, \ldots, n\}$ generate the group of 1-parameter subgroups, it follows that the $T$-action on $M$ admits enough limits.
  
\end{rema}

\section{Even dimensional complex quadrics and its GKM graph}\label{edcqigg}
In this section we define even dimensional complex quadrics and briefly recall their GKM graphs from \cite{ku23}. We broadly follow the notation and conventions from that article.

\begin{defe}\label{compquaeven}
	An even dimensional complex quadrics is defined to be the following manifold \[Q_{2n}:=\Big\{[x_1:\cdots:x_{2n+2}]\in\mathbb{CP}^{2n+1}~|~\sum_{j=1}^{n+1}x_ix_{2n+3-i}=0\Big\}.\]
\end{defe}

 There exists a natural action of $T^{n+1}$ on $Q_{2n}$ : \[Q_{2n}\times T^{n+1} \longrightarrow Q_{2n} \]
   \begin{equation}\label{torusaction}
  	[x_1:\cdots:x_{2n+2}]\cdot(t_1,\ldots,t_{n+1}):=[x_1t_1:\cdots:x_{n+1}t_{n+1}:t_{n+1}^{-1}x_{n+2}:\cdots:t_1^{-1}x_{2n+2}].
  \end{equation}
  Note that, $Q_{2n}\cong SO(2n+2)/(SO(2n)\times SO(2))$. Hence the above action is equivalent to a maximal torus $T^{n+1}$- action (of $SO(2n+2)$),  induced by restricting  the
  	$SO(2n+2)$- action  by left multiplication. 
   	 Additionally, $T^{n+1}$ is also a maximal torus of $SO(2n)\times SO(2)$. Therefore from \cite{ghz06}, the fixed points and one-dimensional orbits under  $T^{n+1}$- action have the structure of a graph.   
  
Since, the above action \eqref{torusaction} is not effective due to its non-trivial kernel $\bz_2=\{\pm1\}$, {we consider $T^{n+1}/\bz_2~(\simeq T^{n+1})$- action on $Q_{2n}$.} In particular, we obtain the GKM graph of $Q_{2n}$ (denoted by $\mathscr{Q}_{2n})$ by labeling edges with tangential representations on (effective) $T^{n+1}$-fixed points.

\begin{rema}
Throughout the paper, we consider the cohomology ring \begin{equation}H^*(BT^{n+1})\cong \bz[x_1,\ldots,x_{n+1}].\end{equation} The generators $x_k$, for $k=1, \ldots, n+1$, has degree 2 and can be regarded as the $k$-th projection $\rho_k:T^{n+1}\to S^1$. In particular, we commonly use the following equivalences :
\[H^2(BT^{n+1})\cong (\mathfrak{t}^{n+1}_\bz)^*\cong \mbox{Hom}(T^{n+1}, S^1)\cong \bz^{n+1}.\]
Further, we  consider the representation ring of $T^{n+1}$ as the following Laurent polynomial ring 
    \begin{equation}\label{oldrepring}
	R(T^{n+1})\cong \bz[y_1^{\pm1},\ldots,y_{n+1}^{\pm1}],
	\end{equation} 
generated by 1-dimensional representations $y_k=e^{2\pi\sqrt{-1}x_k}, ~y_k^{-1}=e^{-2\pi\sqrt{-1}x_k}$ for all $k=1,\ldots,n+1$. For further details, interested readers are referred to \cite[Section 4.3]{athir}.

\end{rema}

  \subsection{The GKM graph of the  $T^{n+1}$-action on $Q_{2n}$} We first recall the GKM graph of the non-effective  $T^{n+1}$ action \eqref{torusaction} on $Q_{2n}$. Note that the graph for the effective $T^{n+1}$-action remains the same, with a modified axial function (see. \eqref{newaxial}). 
 
  One can check that the fixed points of $Q_{2n}$ are \[Q_{2n}^T=\{[e_i]\,:\,1\leq i\leq 2n+2\},\] where $[e_i]=[0:\cdots:0:1:0:\cdots:0]\in \mathbb{CP}^{2n+1}$ which has 1 in the $i$-th place and 0 elsewhere. Furthermore, the invariant two spheres of $Q_{2n} $ are denoted by the following symbol :
\begin{equation}
	[w_i:w_j]:= [0:\cdots:0:w_i:0:\cdots:0:w_j:0:\cdots:0]
\end{equation}
where $i+j\neq 2n+3$. Note that $[w_i:w_j]$, where $i+j= 2n+3$, is one of the fixed points $[e_i]$ or, $[e_j]$ (see Definition \ref{compquaeven}).

Therefore, we associate the following graph $\Gamma_{2n} := (\mathcal{V}_{2n},~ \mathcal{E}_{2n})$ with the pair $(Q_{2n}, T^{n+1})$ where the vertices are $Q_{2n}^T$ i.e \[\mathcal{V}_{2n}=[2n+2]:=\{1,2,\ldots,2n+2\},\] 
and, $[e_i], [e_j]\in Q_{2n}^T$ are connected by an edge if and only if there is an $T^{n+1}$-invariant embedded 2-sphere  containing $[e_i], [e_j]$ as its poles. Hence the set of edges of $\Gamma_{2n}$ is given by \[\mathcal{E}_{2n}:=\{ij\,:\,i,j\in[2n+2]\text{ such that } i\neq j,\,i+j\neq 2n+3\}.\]

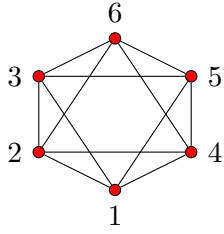
\begin{figure}[H]
	\centering
\begin{tikzpicture}[row sep=small, column sep=small]
 \node[draw, circle, fill=red, inner sep=1.5pt, label=right:5] (A) at (1, 0.5) {};
	\node[draw, circle, fill=red, inner sep=1.5pt, label=right:4] (B) at (1,-0.5) {};
	\node[draw, circle, fill=red, inner sep=1.5pt, label=left:3] (C) at (-1, 0.5) {};
    \node[draw, circle, fill=red, inner sep=1.5pt, label=left:2] (D) at (-1, -0.5) {};
	\node[draw, circle, fill=red, inner sep=1.5pt, label=6] (E) at (0, 1) {};	
	\node[draw, circle, fill=red, inner sep=1.5pt, label=below:1] (F) at (0, -1) {};		
	\draw (A) -- node[midway, below, sloped]  {} (B);
	\draw (B) -- node[midway, below right, sloped] {} (D);
	\draw (C) -- node[midway, below left, sloped] {} (A);
	\draw (A) -- node[midway, below, sloped]  {} (E);
	\draw (A) -- node[midway, below, sloped]  {} (F);
	\draw (B) -- node[midway, below, sloped]  {} (F);
	\draw (B) -- node[midway, below, sloped]  {} (E);
	\draw (F) -- node[midway, below, sloped]  {} (D);
		\draw (F) -- node[midway, below, sloped]  {} (C);
			\draw (D) -- node[midway, below, sloped]  {} (C);
				\draw (C) -- node[midway, below, sloped]  {} (E);
	\draw (D) -- node[midway, below, sloped]  {} (E);
\end{tikzpicture}

  \caption{The above graph is $\Gamma_4$ for $n=2$ induced from the $T^3$ -action.}
\label{fig:gamma4diag}
\end{figure}

\begin{rema}
For simplicity, we shall write $j$ as $\overline{i}$ when $i+j=2n+3$. Hence the set of vertices and edges in $\Gamma_{2n}$ can be re-written as 
\begin{equation}
\mathcal{V}_{2n}=\{1,2,\ldots,n+1,\overline{n+1},\overline{n+2},\ldots, \overline{1}\}. 
	\end{equation}
\begin{equation}
\mathcal{E}_{2n}:=\{ij\,:\,i,j\in\mathcal{V}_{2n}\text{ such that } j\neq i,\overline{i}\}.
	\end{equation}	\end{rema}
In order to define an axial function $\widehat{\alpha}  :\ce_{2n}\to H^2(BT^{n+1})$, it is enough to compute the tangential representation of each invariant 2-spheres $[w_i:w_j]\in Q_{2n}$. Note that, an element $t=(t_1,\ldots,t_{n+1})\in T^{n+1}$ acts on  $[w_i:w_j]$ by the following : \[ [w_i:w_j]\mapsto[\rho_i(t)w_i:\rho_j(t)w_j],\]
where $\rho_i: T^{n+1}\to S^1$ is the surjective homomorphism defined by 
\[\rho_i(t)=\begin{cases}
	t_i~~~~~~~&\text{if $i\in\{1,\ldots,n+1\}$},\\
	t_{\overline{i}}^{-1} ~~~~~~~&\text{if $i\in\{n+2,\ldots,2n+2\}$}.
\end{cases}\]
Therefore,  the tangential representations around the fixed points of $[w_i:w_j]$  are given by
\[[1:w_j]\mapsto [1:\rho_i(t)^{-1}\rho_j(t)w_j],~~~~[w_i:1]\mapsto [\rho_j(t)^{-1}\rho_i(t)w_1:1].\]
Consequently, we define the axial function $\widehat{\alpha} : \ce_{2n}\to H^2(BT^{n+1})$  by the equation \begin{equation}\widehat{\alpha}(ij)=x_j-x_i,\end{equation} where $x_i\in  H^2(BT^{n+1})$ corresponds to the $i$-th coordinate projection $\rho_i : T^{n+1}\to \bc\setminus \{0\}$ and satisfies the following :
\begin{itemize}
	\item $\{x_i : i\in [n+1]\}$ generates $H^2(BT^{n+1})$.
	\item for $i\in\{n+2,\ldots, 2n+2\}$, $x_i:=-x_{\overline{i}}$.
\end{itemize}

However, since the $T^{n+1}$ action on $Q_{2n}$ in \eqref{torusaction} is not effective, it can be verified that for any vertex $i_0\in\mathcal{V}_{2n}$, the collection \begin{equation}\label{axialvalues}
	\{\widehat{\alpha}(i_0j)~|~i_0+j\not=2n+3\}
	\end{equation} 
does not span $H^2(BT^{n+1})$. In particular, $\widehat{\alpha}$ is not an \emph{effective} axial function (see, \cite[Section 2.1]{k19}).

For example, if we consider $1\in\cv_{2n}$, the vectors in \eqref{axialvalues} i.e
\begin{equation}\label{axialvalues1}
	x_2-x_1,\ldots,x_{n+1}-x_1,-x_{n+1}-x_1,\ldots,-x_2-x_1
\end{equation}
span the lattice $\langle 	x_2-x_1,\ldots,x_{n+1}-x_1,-x_{n+1}-x_1\rangle_\bz$, which is a proper subspace of $H^2(BT^{n+1})$. In order to  avoid this problem, we will identify $\langle 	x_2-x_1,\ldots,x_{n+1}-x_1,-x_{n+1}-x_1\rangle_\bz$ with  $H^2(BT^{n+1})$. This defines the axial function induced by the effective $T^{n+1}$- action on $Q_{2n}$.
In particular, we replace \begin{itemize}
	\item $x_i-x_1$ by $x_{i-1}$ for all $i=2,\ldots,n+1$.
	\item $-x_{n+1}-x_1$ by $x_{n+1}$.
\end{itemize}
Therefore, we can rewrite the elements in \eqref{axialvalues1} as the following elements respectively
\begin{equation}\label{newaxialvalues1}
	x_1,\ldots,x_n,x_{n+1},-x_{n-1}+x_n+x_{n+1},\ldots,-x_1+x_n+x_{n+1},
\end{equation}
which spans $H^2(BT^{n+1})$. {The similar procedure applies to the other vertices as well.}
\begin{rema}
	Note that the above identification depends on a choice of the ordering of the vertices of the GKM graph.
\end{rema}
\begin{defe}\label{defh}
Set $h : \cv_{2n}\to H^2(BT^{n+1})$ as 
\begin{equation}\label{defhj}
	h(j)=\begin{cases}
	x_{j-1}-x_{n+1} \text{ if } j=1,\ldots,n+2,\\
	x_{n}-x_{2n+2-j} \text{ if } j=n+3,\ldots,2n+2,
\end{cases}\end{equation} where $x_0=0$.
And, we define  the axial function  $\alpha : \ce_{2n}\to  H^2(BT^{n+1})$ by \begin{equation}\label{newaxial}
	\alpha(ij)=h(j)-h(i)\text{~~~~ for $j\neq \{i,\overline{i}\}$.}  \end{equation}
\end{defe}
From the above definition, the following holds,
\begin{lema}{\cite[Lemma 2.6]{ku23}}
For every vertex $i\in\cv_{2n}$, and for any distinct three vertices $k_1,k_2,k_3\in\cv_{2n}\setminus\{i,\overline{i}\}$, the axial functions $\alpha(ik_1),~\alpha(ik_2),~\alpha(ik_3)$ are linearly independent. In other words, the GKM graph $(\Gamma_{2n}, \alpha)$ is three-independent.
\end{lema}
\begin{rema}
The GKM graph $(\Gamma_{2n}, \alpha)$ is  induced by the effective $T^{n+1}$-action defined in \eqref{eqref3.17}. Note that the vectors in  \eqref{axialvalues1} induces the following homomorphism $\psi : T^{n+1}\to T^{2n+2}$ :
		\[(t_1,\ldots,t_{n+1})\mapsto (t_1t_1^{-1},t_2t_1^{-1},\ldots,t_{n+1}t_1^{-1},t_{n+1}^{-1}t_1^{-1},\ldots,t_2^{-1}t_1^{-1},t_1^{-2})\]
with \[\begin{split}
ker~\psi=&\{\pm1\}\\ im~\psi=&\{(t_1t_1^{-1},t_2t_1^{-1},\ldots,t_{n+1}t_1^{-1},t_{n+1}^{-1}t_1^{-1},\ldots,t_2^{-1}t_1^{-1},t_1^{-2}) : t_1,\ldots,t_{n+1}\in T^{n+1}\}.\\
=&\{(1,r_1,\ldots,r_{n+1},r_{n-1}^{-1}r_nr_{n+1},\ldots, r_{1}^{-1}r_nr_{n+1}, r_nr_{n+1})~ :~ r_1,\ldots,r_{n+1}\in T^{n+1}\}
\end{split}\]
where $r_i=t_{i+1}t_1^{-1}$ for $i=1,\ldots,n$ and $r_{n+1}=t_{n+1}^{-1}t_1^{-1}$. Further, by the first isomorphism theorm, we have the following  \[im~\psi\simeq T^{n+1}/ker~\psi\simeq T^{n+1}/ \bz_2\simeq T^{n+1}.\]
Note that the standard $T^{2n+2}$-action of $Q_{2n}$ is as follows :
\[(t_1,\ldots,t_{2n+2})\cdot[x_1:\cdots:x_{2n+2}]:=[t_1x_1:\cdots:t_{2n+2}x_{2n+2}]\] for $(t_1,\ldots,t_{2n+2})\in T^{2n+2},~~[x_1:\cdots:x_{2n+2}]\in Q_{2n}$. The effective $T^{n+1}$-action on $Q_{2n}$ is in particular the restricted action of $im~\psi\subset T^{2n+2}$.
\begin{equation}\label{eqref3.17}
[x_1:\cdots:x_{2n+2}]\mapsto [x_1:r_1x_2:\cdots:r_{n+1}x_{n+2}:r_{n-1}^{-1}r_nr_{n+1}x_{n+3}:\cdots:r_nr_{n+1}x_{2n+2}]
\end{equation}
One can observe that the axial function of the above action around $1\in\cv_{2n}=Q_{2n}^T$ coincides with \eqref{newaxialvalues1}.
\end{rema}

\begin{rema}
	We also make the following replacements in \eqref{oldrepring}  \begin{itemize}
		\item $y_iy_1^{-1}$ by $y_{i-1}$ for all $i=2,\ldots,n+1$.
		\item $y_{n+1}^{-1}y_1^{-1}$ by $y_{n+1}$.
	\end{itemize}
\end{rema}
\begin{rema}
	Set $f:\cv_{2n}\to R(T^{n+1})$ by 
	\begin{equation}\label{deffj}
	f(j)=e^{2\pi\sqrt{-1}h(j)} ~~~\text{	for  $j\in\cv_{2n}$.}
	\end{equation}
\end{rema}
Therefore, for $j=1,\ldots,n+1$,
\begin{equation}\label{yjwithfj}
	f(j+1)f(1)^{-1}=e^{2\pi\sqrt{-1}x_j}=y_j.
\end{equation}

\begin{notn}
	Hereafter, we denote the GKM graph $(\Gamma_{2n},\alpha)$ by $\mathscr{Q}_{2n}$ and, the $K$- ring $K_\alpha(\Gamma_{2n})$ (see, Definition \ref{kclass})  by $K(\mathscr{Q}_{2n})$.
\end{notn}

\section{Generators and relations of $K(\mathscr{Q}_{2n})$}\label{grkring}
In this section, we introduce two types of $K$- classes of $\mathscr{Q}_{2n}$ and four types of relations among them. In Section \ref{secmainth}, we show that the abovementioned $K$- classes generate $K(\mathscr{Q}_{2n})$.
\subsection{Generators of type I}
\begin{defe}\label{defmv}
	For a vertex $v\in\mathcal{V}_{2n}$, we define the  function $M_v: \mathcal{V}_{2n}\to R(T^{n+1})$  by \[M_v(l)=
	\begin{cases}
 1 ; &\text{ if $l=v$},\\
 e^{2\pi\sqrt{-1}\alpha(\overline{v}k)}\cdot e^{2\pi\sqrt{-1}\alpha(\overline{v}\,\overline{k})}=y_ny_{n+1}^{-1}f(\overline{v})^{-2}; &\text{  if $l=\overline{v}$},\\
e^{2\pi\sqrt{-1}\alpha(lv)}=f(v)f(l)^{-1}; &\text{  if $l\neq v,\overline{v}$}.
\end{cases}\]
Further, we define $M_v^{-1}: \mathcal{V}_{2n}\to R(T^{n+1})$ by \[M_v^{-1}(l):=(M_v(l))^{-1}.\]
\end{defe}
\begin{rema}
	Note that the above definition does not depend on the choice of $k\in \cv_{2n}\setminus \{v,\overline{v}\}$ since for any $k\in \cv_{2n}\setminus\{v,\overline{v}\}$, we have \[\alpha(\overline{v}k)+\alpha(\overline{v}\,\overline{k})=h(k)-h(\overline{v})+h(\overline{k})-h(\overline{v})=h(k)+h(\overline{k})-2h(\overline{v}).\]
	Further, from Definition \ref{defh}, \begin{equation}\label{hiibar}
	h(k)+h(\overline{k})=x_n-x_{n+1}.
	\end{equation} In particular, 
	\begin{equation}
		e^{2\pi\sqrt{-1}\alpha(\overline{v}k)}\cdot e^{2\pi\sqrt{-1}\alpha(\overline{v}\,\overline{k})}=y_ny_{n+1}^{-1}~e^{-4\pi\sqrt{-1}h(\overline{v})}=y_ny_{n+1}^{-1}f(\overline{v})^{-2}.
	\end{equation}
\end{rema}

\begin{propo}
	For any $v\in \mathcal{V}_{2n}$, $M_v\in K(\mathscr{Q}_{2n})$.
\end{propo}

\begin{proof}
From Definition \ref{kclass}, it is enough to show that for any edge $(jk)\in\mathcal{E}_{2n}$ (i.e $k\not=j,\overline{j}$), we have \[M_v(j)-M_v(k)=0\,\,(\text{mod } 1-e^{2\pi\sqrt{-1}\alpha(jk)}).\]

We verify the above condition on a case-by-case basis :
\begin{enumerate}
	\item  \underline{If $j=v$ :} For any $k\not=v,\overline{v}$, we have 
\[M_v(j)-M_v(k)=1-(e^{2\pi\sqrt{-1}\alpha(kj)})\equiv 0\,\,(\text{mod } 1-e^{2\pi\sqrt{-1}\alpha(jk)}),\]
since $\alpha(kj)=-~\alpha(jk)$ from \eqref{newaxial}.

\item  \underline{If $j=\overline{v}$ :}  For  $k\not=v,\overline{v}$, we have 
\[\begin{split}M_v(j)-M_v(k)
	&=e^{2\pi\sqrt{-1}\alpha(\overline{v}k)}e^{2\pi\sqrt{-1}\alpha(\overline{v}\,\overline{k})}-e^{2\pi\sqrt{-1}\alpha(kv)}\\
	&=e^{2\pi\sqrt{-1}\alpha(\overline{v}k)}e^{2\pi\sqrt{-1}\alpha(\overline{v}\,\overline{k})}-e^{2\pi\sqrt{-1}\alpha(\overline{v}\overline{k})}~~~~~\text{(see, Remark \ref{alphavbarkbar})}\\
	&=\big(e^{2\pi\sqrt{-1}\alpha(\overline{v}k)}-1\big)e^{2\pi\sqrt{-1}\alpha(\overline{v}\,\overline{k})}\equiv 0\,\,(\text{mod } 1-e^{2\pi\sqrt{-1}\alpha(jk)}).
	\end{split}\]
	
	\item \underline{If $j\not=v,\,\overline{v}$ :} We break this case into two more cases.\\
	
	\begin{enumerate}
		\item \underline{If $k=\overline{v}$ or $k={v}$:} We follow the similar arguements as above.

		\item \underline{If $k\not=v,\overline{v}$ :}  We have 
		\[\begin{split}M_v(j)-M_v(k)
			&=e^{2\pi\sqrt{-1}\alpha(jv)}-e^{2\pi\sqrt{-1}\alpha(kv)}\\
			&=e^{2\pi\sqrt{-1}\alpha(kv)}(e^{2\pi\sqrt{-1}\alpha(jk)}-1) ~~~~~\text{[as, $\alpha(jv)=\alpha(kv)+\alpha(jk)$]}\\
			&\equiv 0 ~~~~(\text{mod } 1-e^{2\pi\sqrt{-1}\alpha(jk)}).
\end{split}\]
\end{enumerate}
\end{enumerate}
Therefore, 	 $M_v\in K(\mathscr{Q}_{2n})$ for any $v\in \mathcal{V}_{2n}$.
\end{proof}

\begin{rema}\label{alphavbarkbar}
	Note that \[
	\begin{split}
	\alpha(\overline{v}\overline{k})&=h(\overline{k})-h(\overline{v})=(x_n-x_{n+1}-h(k))-(x_n-x_{n+1}-h(v))~~~~~ \text{(see, \eqref{hiibar})}\\
	&=	h(v)-h(k)= \alpha(kv).
	\end{split}\]
\end{rema}

\begin{rema}
Similarly, one can show $M_v^{-1}\in K(\mathscr{Q}_{2n})$, for any $v\in \mathcal{V}_{2n}$.
\end{rema}

\begin{propo}\label{genrepring}
	The generator $y_i\in R(T^{n+1})$ for each $i=1,\ldots,n+1$ satisfies the following equality \[y_i=M_{i+1}M_1^{-1}.\]
\end{propo}
\begin{proof}
For $i\in\{1,\ldots,n+1\}$, we evaluate 	$M_{i+1}M_1^{-1}(j)=M_{i+1}(j)(M_1(j))^{-1}$  for all $j\in\cv_{2n}$.

If $j\in \cv_{2n}\setminus\{\overline{1},\overline{i+1}\}$,
\[\begin{split}
M_{i+1}M_1^{-1}(j)&=f(i+1)f(j)^{-1}	(f(1)f(j)^{-1})^{-1}  ~~~~(\text{by Definition \ref{defmv}})\\&=f(i+1)f(1)^{-1}=y_i,~~~~(\text{see, \eqref{yjwithfj}}) 
\end{split}\]

If $j= \overline{1}=2n+2$,
\[\begin{split}
M_{i+1}M_1^{-1}(2n+2)&=f(i+1)f(2n+2)^{-1}	(y_ny_{n+1}^{-1}f(2n+2)^{-2})^{-1}~~~~(\text{by Definition \ref{defmv}})\\
&=f(i+1)f(2n+2)\big(y_ny_{n+1}^{-1}\big)^{-1}\\
&=f(i+1)f(2n+2)\big(f(i+1)f(\overline{i+1})\big)^{-1} ~~~~(\text{see \eqref{deffj}, \eqref{hiibar}})\\
&=f(2n+2)f(\overline{i+1})^{-1}=y_i ~~~~(\text{see \eqref{deffj}, \eqref{defhj}})\\
\end{split},\]

If $j= \overline{i+1}=2n+2-i$,
\[\begin{split}
	M_{i+1}M_1^{-1}(2n+2-i)&=y_ny_{n+1}^{-1}f(2n+2-i)^{-2} (f(1)f(2n+2-i)^{-1})^{-1}\\
	&=y_nf(2n+2-i)^{-1} ~~~~~~~~\text{ (as, $f(1)=y_{n+1}^{-1}$)}\\
	&=y_i         
\end{split},\]

Note that $f(2n+2-i)=\begin{cases}
	y_n\cdot y_i^{-1}; ~~~~~~&i=1,\ldots,n-1,\\
	y_{2n+1-i}\cdot y_{n+1}^{-1} ~~~ &i=n,n+1.
\end{cases}$

Hence, the proposition follows.
\end{proof}

\begin{exam} For $n=2$, we consider $\mathscr{Q}_4$ (see, Figure \ref{fig:gamma4diag}) and $1\in \cv_4$. Then $M_1\in K(\mathscr{Q}_4) $ is given by 
\[	\begin{split}
M_1(1)= 1, ~~~~~~~~~~~~~~~~~~~~~~~~~~~~~~~~~&M_1(2)=f(1)f(2)^{-1}=y_1^{-1}\\
M_1(3)=f(1)f(3)^{-1}=y_2^{-1},~~~~~~~~~~~~~&M_1(4)=f(1)f(4)^{-1}=y_3^{-1}\\
M_1(5)=f(1)f(5)^{-1}=y_1y_2^{-1}y_3^{-1}, ~~~~~~ &M_1(6)=M_1(\overline{1})=y_2y_3^{-1}f(\overline{1})^{-2}= y_2^{-1}y_3^{-1}.
	\end{split}\]	
\end{exam}

\subsection{Generators of type II}

\begin{defe}\label{delta_P}
Let $P$ be any non-empty subset of $\mathcal{V}_{2n}=[2n+2]$ such that the full subgraph $\Gamma_P$ is a complete subgraph of $\Gamma_{2n}$ or  equivalently$^1$,\footnote[1]{$^1$Notice that these are equivalent due to the definition of $\Gamma_{2n}$} $\{i,\overline{i}\}\not\subset P$ for all $i\in \mathcal{V}_{2n}$.

 Then we define the  map $\Delta_P:\cv_{2n}\to R(T^{n+1})$ by
\begin{equation}\label{delta_k}
	\Delta_P(l)=\begin{cases}
		\displaystyle	\prod_{k\notin P\cup \{\overline{l}\}}\big(1-e^{2\pi\sqrt{-1}\alpha(lk)}\big)=	\prod_{k\notin P\cup \{\overline{l}\}}(1-f(k)f(l)^{-1}), & \text{if $l\in P$,}\\
		~~~~~~~~~~~~~~0 &\text{if $l\notin P$.}
\end{cases}\end{equation}
\end{defe} 
 
\begin{rema}
	Note that $\Delta_P$ in \eqref{delta_k} is basically the \emph{Thom class} (equivariant $K$- theoretic) of the GKM subgraph $\Gamma_P$. Geometrically, the Thom class (equivariant $K$- theoretic) of a codimension $2m$ GKM submanifold $N$ of a GKM manifold $Y$, for effective $T$-action, with an almost complex structure can be defined as follows :

	Consider the normal bundle $\nu$ of the submanifold $N$, and define its \emph{Thom space} as $Th(\nu):= Y/D(\nu)^c$, where $D(\nu)^c$ denotes the complement of the unit disc bundle embedded in $Y$. Since $\nu$ inherits its orientation from the almost complex structure (in fact, this becomes a complex $m$-dimensional vector bundle), we have the following isomorphism known as the \emph{Thom isomorphism} :
	\[K_T(N)\to K_T(Th(\nu))\]
	On the other hand, we have the homomorphism $K_T(Th(\nu))\to K_T(Y)$, induced from the collapsing map $Y\to Th(\nu)$. Therefore, we obtain the following homomorphism by composing the two aforementioned homomorphisms :
	\[\psi_N: K_T(N)\to K_T(Y)\]	
	Then, we can define the Thom class of the codimension $2m$ GKM submanifold $N$ by
	\[\psi_N({1})=\tau_N\in K_T(Y).\] Because of the definition of a GKM
	graph, the GKM subgraph $\Gamma_P\subset \mathscr{Q}_{2n}$ is induced from some GKM submanifold $N$ of $Q_{2n}$; therefore, the Thom class $\Delta_P$ of $\Gamma_P$ is the combinatorial analogue of the Thom class $\tau_N$ of $N$.
\end{rema}

\begin{lema}
Let $P\subset \mathcal{V}_{2n}$ such that $\{i,\overline{i}\}\not\subset P$ for all $i\in \mathcal{V}_{2n}$. Then \begin{equation}\Delta_P\in K(\mathscr{Q}_{2n}).\end{equation}
\end{lema}
\begin{proof}
	Let $e=(ij)\in\mathcal{E}_{2n}$ be an edge in $\Gamma_{2n}$ (i.e $j\neq \overline{i},~i\neq \overline{j}$ ) for $i,j\in\mathcal{V}_{2n}$. We verify the condition \eqref{kclasscond} on a case-by-case basis : \\
	\begin{enumerate}[(i)]
		\item  If  $i,j\notin K$, then  $\Delta_P(i)=\Delta_P(j)=0$. Hence \[\Delta_P(i)-\Delta_P(j)=0\,\,\,(mod\,\, 1-e^{2\pi\sqrt{-1}\alpha(ij)}).\] 
		
		\item If  $i\in K$, $j\notin K$, then from \eqref{delta_k}, we have \[\displaystyle\Delta_P(i)-\Delta_P(j)=\prod_{k\notin P\cup \{\overline{i}\}}\big(1-e^{2\pi\sqrt{-1}\alpha(ik)}\big)-0=0\,\,\,(mod\,\, 1-e^{2\pi\sqrt{-1}\alpha(ij)}).\]
		The last equality holds as $j\notin P\cup \{\overline{i}\}$. This implies $(1-e^{2\pi\sqrt{-1}\alpha(ij)})$ is a factor of $\Delta_P(i)$.\\
		
		\item Finally, we assume $i,j\in P$. Let $e'=(ik)$ be an edge with initial point $i$ such that $k\notin P$. Therefore from \eqref{delta_k}, we deduce that $\big(1-e^{2\pi\sqrt{-1}\alpha(ik)}\big)$ divides $\Delta_P(i)$.
		
		\noindent
		Since $P$ is invariant under the connection $\nabla_e$ at $e=(ij)$, we have $\nabla_{e}(ik)$ is not an edge in $\Gamma_P$ i.e the terminal point of $\nabla_{e}(ik)$ is not in $P$. Therefore, $\big(1-e^{2\pi\sqrt{-1}\alpha(\nabla_{e}(ik))}\big)$ is one of the factors in $\Delta_P(j)$.
		
		\noindent Furthermore, from the third property of Definition \ref{intaxialfn},  we have  the following  
		\begin{equation}
			1-e^{2\pi\sqrt{-1}\alpha(\nabla_{e}(e'))}\equiv 1-e^{2\pi\sqrt{-1}\alpha(e')} ~~~(\text{mod}~1-e^{2\pi\sqrt{-1}\alpha(e)}). 	\end{equation} 
Since, one can have the similar relations for every other factors in $\Delta_P(i)$, we obtain \[\Delta_P(i)-\Delta_P(j)\equiv0\,\,\,(mod\,\, 1-e^{2\pi\sqrt{-1}\alpha(ij)}).\]
	\end{enumerate}

Hence, the lemma follows.	
\end{proof}

\begin{exam} For the GKM graph $\mathscr{Q}_4$ and $P=\{2,4,6\}$ such that $\Gamma_P$ is a complete subgraph (see, Figure \ref{fig:gamma4diag}), $\Delta_P\in K(\mathscr{Q}_4) $ is given by 
	\[	\begin{split}
\Delta_P(2) 
=(1-y_1^{-1})(1-y_2y_1^{-1}), ~~~~~~~~&\Delta_P(4) 
=(1-y_3^{-1})(1-y_2y_1^{-1}),\\
\Delta_P(6)= (1-y_3^{-1})(1-y_1^{-1}),
 ~~~~~~~~~&\Delta_P(1)=\Delta_P(3)=\Delta_P(5)=0.
	\end{split}\]	
\end{exam}

\subsection{Relations  between $M_v$ and $\Delta_P$'s}\label{relations}
Within this section, we present four types of relations  between $M_v$ and $\Delta_P$'s.

 In the results below, we frequently use $1\in R(T^{n+1})$ as $1:\cv_{2n}\to R(T^{n+1})$ (see \eqref{algebramap}) which takes any vertices $v\in\cv_{2n}$ to $1\in R(T^{n+1})$.
\subsubsection{Relation 1}
Let $J$ be a subset of $\mathcal{V}_{2n}$. We consider 
\begin{equation}\label{F_J}
	F_J:=\begin{cases}
		&1-M_v\,    \text{ if $J=\mathcal{V}_{2n}\setminus\{v\}$ for a vertex $v\in \mathcal{V}_{2n}$},\\
		& \Delta_J\, \text{ if $\{i,\overline{i}\}\not\subset J$ for every $i\in\mathcal{V}_{2n}$}.
	\end{cases}
\end{equation}

\begin{lema}\label{relation1}
	The following relation holds :
	\begin{equation}
		\displaystyle\prod_{\cap J=\emptyset} F_J=0.
	\end{equation}
\end{lema}
\begin{proof}
The lemma follows from Definition \ref{defmv} and Definition \ref{delta_P}.
\end{proof}

\subsubsection{Relation 2}
Let $I\subset\cv_{2n}=[2n+2]$ be a subset with $|I|=n$  such that the full subgraph $\Gamma_I$ is a complete subgraph of $\Gamma_{2n}$. Hence, there exists an unique pair $\{b,\overline{b}\}\in I^c=\cv_{2n}\setminus I$ such that \[\Delta_P,\,\Delta_L\in K(\mathscr{Q}_{2n}),\] where $P=(I\cup \{b\})^c=I^c\setminus\{b\}$ and $L=(I\cup\{\overline{b}\})^c=I^c\setminus\{\overline{b}\}$. 
\begin{lema}\label{relation3}
Let	$I\subset\cv_{2n}=[2n+2]$ be a subset which satisfies the above conditions. Then the following holds :
\begin{equation}\label{relation3eqn}
	\prod_{i\in I}(1-M_i)=\Delta_{(I\cup \{b\})^c}+M_b\cdot \Delta_{(I\cup \{\overline{b}\})^c}.\end{equation}
	\end{lema}
\begin{proof}
We verify the claim for every vertices in $\cv_{2n}$.

	Let $v\in I$. Hence, $v\notin (I\cup \{b\})^c, (I\cup \{\overline{b}\})^c$. Therefore from \eqref{delta_k}, \[\Delta_{(I\cup \{b\})^c}(v)=0= \Delta_{(I\cup \{\overline{b}\})^c}(v).\] Also, $\displaystyle\prod_{i\in I}(1-M_i)(v)=0$ as $(1-M_v)$ is a factor of LHS in \eqref{relation3eqn}.

	For the vertex $b~(\notin I)$, we have 
\[	\begin{split}
	 	\Big(\Delta_{(I\cup \{b\})^c}+M_b\cdot \Delta_{(I\cup \{\overline{b}\})^c}\Big)(b)&=0+M_b(b)\cdot \Delta_{(I\cup \{\overline{b}\})^c}(b)\\
	 	&=0+1\cdot \prod_{i\in I}\big(1-e^{2\pi\sqrt{-1}\alpha(bi)}\big)=\prod_{i\in I}(1-M_i)(b).	 	
	 	\end{split}\]
	 
	 For the vertex $\overline{b}~(\notin I)$, we have 
	 	\[\begin{split}
	 		\Big(\Delta_{(I\cup \{b\})^c}+M_b\cdot \Delta_{(I\cup \{\overline{b}\})^c}\Big)(\overline{b})&=\Delta_{(I\cup \{b\})^c}(\overline{b})+M_b(\overline{b})\cdot 0\\
	 		&=\prod_{i\in I}\big(1-e^{2\pi\sqrt{-1}\alpha(\overline{b}i)}\big)+0=\prod_{i\in I}(1-M_i)(\overline{b}).	 	
	 	\end{split}\]
	 	
If $v\notin I\cup \{b, \overline{b}\}$, then $\overline{v}\in I$. Thus we have
	 		\[	\begin{split}
	 			\prod_{i\in I}(1-M_i)(v)&=(1-M_{\overline{v}})(v)\cdot\prod_{i\in I\setminus\{\overline{v}\}}(1-e^{2\pi\sqrt{-1}\alpha(vi)})\\
	 			&=(1-e^{2\pi\sqrt{-1}\alpha(vb)}e^{2\pi\sqrt{-1}\alpha(v\,\overline{b})})\cdot\prod_{i\in I\setminus\{\overline{v}\}}(1-e^{2\pi\sqrt{-1}\alpha(vi)})\\
	 			&=\big((1-e^{2\pi\sqrt{-1}\alpha(vb)})+e^{2\pi\sqrt{-1}\alpha(vb)}(1-e^{2\pi\sqrt{-1}\alpha(v\,\overline{b})})\big)\cdot\prod_{i\in I\setminus\{\overline{v}\}}(1-e^{2\pi\sqrt{-1}\alpha(vi)})\\
	 			&=\prod_{i\in (I\setminus\{\overline{v}\})\cup\{b\}}(1-e^{2\pi\sqrt{-1}\alpha(vi)}) +e^{2\pi\sqrt{-1}\alpha(vb)}\cdot \prod_{i\in (I\setminus\{\overline{v}\})\cup\{\overline{b}\}}(1-e^{2\pi\sqrt{-1}\alpha(vi)})\\
	 			&=\big(\Delta_{(I\cup \{b\})^c}+M_b\cdot \Delta_{(I\cup \{\overline{b}\})^c}\big)(v)
	 			\end{split}\]
	 			Hence, the lemma follows.
\end{proof}

\subsubsection{Relation 3} 
\begin{lema}\label{relation4}Fix $i\in\cv_{2n}$. Let $P\subset \cv_{2n}$ be a subset containing $i$ such that $|P|>1$ and $\Gamma_P$ is a complete subgraph of $\Gamma_{2n}$.
Then we have the following :
\begin{equation}
	\Delta_P\cdot (1-M_i)=\Delta_{P\setminus\{i\}}.
\end{equation}
\end{lema}
	\begin{proof}
		Note that, $\Delta_P\cdot (1-M_i)$ is not zero only on $P\cap (\cv_{2n}\setminus\{i\})=P\setminus\{i\}$. In particular, we have 
		\[\Delta_P\cdot (1-M_i)(v)=\begin{cases}
			\prod_{j\notin(P\setminus\{i\})\cup\{\overline{v}\}}(1-e^{2\pi\sqrt{-1}\alpha(vj)})&\text{ if $v\in P\setminus\{i\}$}\\
				0&\text{ if $v\notin P\setminus\{i\}$}
		\end{cases}\]
		Therefore, the lemma follows from \eqref{delta_k}.
	\end{proof}

\subsubsection{Relation 4}
We define $X:\mathcal{V}_{2n}\to R(T^{n+1})$ by (see, \eqref{deffj})
\[X(k)=y_ny_{n+1}^{-1}f(k)^{-2},\text{\,\, for all $k\in \mathcal{V}_{2n}$}.\]
\begin{lema}\label{relation2}
	For every $v\in \mathcal{V}_{2n}$, we have the following equality :
	\[M_v\cdot M_{\overline{v}}=X.\]
\end{lema}
\begin{proof}
The lemma follows from Definition \ref{defmv} and the fact $f(v)f(\overline{v})=y_ny_{n+1}^{-1}$ (see, \eqref{hiibar}).
\end{proof}
\begin{rema}
	Note that $X\in K(\mathscr{Q}_{2n})$ since $M_v, M_{\overline{v}}\in K(\mathscr{Q}_{2n})$.
\end{rema}

\section{Proof of the main theorem}\label{secmainth} 
In this section, we present the equivariant $K$- ring of  even dimensional complex quadrics in terms of generators and relations. To convey our main theorem  (see, Theorem \ref{mainth}) precisely, we prepare some notation as follows :

Let $\mathbf{M}$ denote the set of $K$-classes \[\{\mathbf{M}_v,~ \mathbf{M}_v^{-1}~:~v\in\cv_{2n}\},\] and, $\mathbf{D}$ denote  the set of $K$-classes  \[\{\mathbf{\Delta}_P~:~\text{$P\subset\cv_{2n}, ~\Gamma_P$ is a complete subgraph of $\Gamma_{2n}$}\}.\]
Let $\bz[\mathbf{M}, \mathbf{D}]$ be the polynomial ring generated by all elements in $\mathbf{M}, \mathbf{D}$. Let $\mathfrak{I}\lhd \bz[\mathbf{M}, \mathbf{D}]$ be an ideal generated by the following types of elements :
 \begin{enumerate}[(i)]
 	\item $\prod_{\cap J=\emptyset} F_J$ for $F_J$ as in \eqref{F_J}.
 		\item {$	\prod_{i\in I}(1-\mathbf{M}_i)-\mathbf{\Delta}_{(I\cup \{b\})^c}-\mathbf{M}_b\cdot \mathbf{\Delta}_{(I\cup \{\overline{b}\})^c}$ for $I$ and $b$ as in Lemma \ref{relation3}.}
 	\item {$\mathbf{\Delta}_P\cdot (1-\mathbf{M}_i)-\mathbf{\Delta}_{P\setminus\{i\}}$ for $\{i\}\subsetneq P$.}
 	  \item $\mathbf{M}_v\mathbf{M}_{\overline{v}}-\mathbf{M}_w\mathbf{M}_{\overline{w}}$ for every distinct $v,w\in\cv_{2n}$.
 \end{enumerate}

We define $\bz[\mathscr{Q}_{2n}]:=\bz[\mathbf{M}, \mathbf{D}]/\mathfrak{I}$. Let $\widetilde{\phi}: \bz[\mathbf{M}, \mathbf{D}]\to K(\mathscr{Q}_{2n})$ be a ring homomorphism which takes $\mathbf{M}_v,~\mathbf{M}^{-1}_v,~\text{and~}\mathbf{\Delta}_P$ to ${M}_v,~{M}^{-1}_v,~\text{and~}{\Delta}_P$  respectively,
 and let  \begin{equation}\label{phi}
	{\phi}: \bz[\mathscr{Q}_{2n}]\to K(\mathscr{Q}_{2n})\end{equation} be the homomorphism induced from $\widetilde{\phi}$.

In other words, the following diagram commutes.
\begin{equation}\label{commdiag}
	\begin{tikzcd}
		&\bz[\mathbf{M}, \mathbf{D}]\arrow[d]	\arrow[dr, "\widetilde{\phi}"] &\\
		&\bz[\mathscr{Q}_{2n}] \arrow[r, "\phi"]&K(\mathscr{Q}_{2n})
			\end{tikzcd}
	\end{equation}
	Now we state the main theorem of this paper,  whose proof relies on Lemma  \ref{surjmainth} and Lemma \ref{injectivitymainth}. 
\begin{guess}\label{mainth}
The homomorphism $\phi$ is an isomorphism of rings, i.e 
	\begin{equation}\label{mainth1}
		\phi: \bz[\mathscr{Q}_{2n}]\xrightarrow{\cong} K(\mathscr{Q}_{2n}).\end{equation} 
	In particular, for the effective $T^{n+1}$- action on $Q_{2n}$, 	\begin{equation}
		\label{mainth2}
		K_{T^{n+1}}(Q_{2n})\cong K(\mathscr{Q}_{2n})\cong  \bz[\mathscr{Q}_{2n}].
		\end{equation}

\end{guess}
In the proofs below, we frequently use $M_i$ and $\Delta_P$ instead of $\widetilde{\phi}(\mathbf{M}_i)$ and $\widetilde{\phi}(\mathbf{\Delta}_P)$, respectively.
\begin{lema}\label{surjmainth}
	The homomorphism $\phi: \bz[\mathscr{Q}_{2n}]\to K(\mathscr{Q}_{2n})$ is surjective.
	\end{lema}

\begin{proof}
This is enough to prove that $\widetilde{\phi}:\bz[\mathbf{M}, \mathbf{D}] \to K(\mathscr{Q}_{2n})$ is surjective. Consider an element $f\in K(\mathscr{Q}_{2n})$. For the vertex  $1\in\cv_{2n}$, one can write $f(1)\in R(T^{n+1})$  (see, Definition \ref{kclass}) as
\[f(1)=\sum_{\mathbf{j}}c_{\mathbf{j}}y_1^{j_1}\cdots y_{n+1}^{j_{n+1}}=h_1\] 
where $c_{\mathbf{j}}\in\bz$ and $\mathbf{j}=(j_1,\ldots, j_{n+1})\in \bz^{n+1}$.

Further, from Proposition \ref{genrepring}, it follows that $M_{i+1}(1)=y_i$ for all $1\leq i\leq n+1$.  Hence 
\[f(1)=\sum_{\mathbf{j}}c_{\mathbf{j}}M_2^{j_1}\cdots M_{n+2}^{j_{n+1}}(1)=h_1.\]
{This means that, there exists an element in $\bz[\mathbf{M}_i, \mathbf{M}_i^{-1}~|~2\leq i\leq n+2]\subset \bz[\mathbf{M}, \mathbf{D}]$ whose image under $\widetilde{\phi}$ coincides with $f(1)$ on the vertex $1\in \cv_{2n}$.}

We next put $f_2=f-h_1$ and hence $f_2(1)=0$. By using the congruence relation \eqref{kclasscond} on the edge $(21)\in\ce_{2n}$, we have
\[f_2(2)-f_2(1)\equiv0~~~\text{mod}~1-e^{2\pi\sqrt{-1}\alpha(21)}=(1-M_1)(2).\]
Therefore, $f_2(2)=h_2(1-M_1)(2)$ for some $h_2\in R(T^{n+1})$.
By Proposition \ref{genrepring}, it follows that $M_{i+1}M_1^{-1}=y_i$ for all $1\leq i\leq n+1$. Hence,  \[h_2\in \bz[M_{i+1}M_1^{-1}, M_{i+1}^{-1}M_1~|1\leq i\leq n+1]\subset \bz[\mathbf{M}, \mathbf{D}] .\]

This shows that $h_2(1-M_1)$ is in the image of $\widetilde{\phi}$.
Put \[f_3=f_2-h_2(1-M_1)~(=f-h_1-h_2(1-M_1)).\]
Notice that \[f_3(1)=0=f_3(2),\] where the first equality follows due to the fact $f_2(1)=0$,  $M_1(1)=1$, and the second equality follows by $f_2(2)=h_2(1-M_1)(2)$. 

 By using the congruence relations on the edge $(31)$ and $(32)\in\ce_{2n}$, we may write
\[f_3(3)=h_3(1-M_1)(1-M_2)(3)\] for some $h_3\in R(T^{n+1})\subset \bz[\mathbf{M}, \mathbf{D}]$. We put \[f_4:=f_3-h_3(1-M_1)(1-M_2)\] which satisfies $f_4(1)=f_4(2)=f_4(3)=0$.
By iterating the similar procedure $n+2$ times, we get an element
\[f_{n+2}:=f_{n+1}-h_{n+1}(1-M_1)\cdots(1-M_n)\] for some $h_{n+1}\in R(T^{n+1})\subset \bz[\mathbf{M}, \mathbf{D}]$ and $f_{n+1}\in K(\mathscr{Q}_{2n})$ satisfying $f_{n+1}(i)=0$ for $1\leq i\leq n$ and $f_{n+1}(n+1)=h_{n+1}(1-M_1)\cdots(1-M_n)(n+1)$.

Furthermore, we have 
\begin{equation}
	\begin{split}\label{surj1}
f_{n+2}&=f_{n+1}-h_{n+1}(1-M_1)\cdots(1-M_n)\\
&=f_n-h_{n}(1-M_1)\cdots(1-M_{n-1})-h_{n+1}(1-M_1)\cdots(1-M_n)\\
&~~~~~~~~\vdots\\
&=f-h_1-h_2(1-M_1)-\cdots-h_{n}(1-M_1)\cdots(1-M_{n-1})-h_{n+1}(1-M_1)\cdots(1-M_n).
\end{split}
\end{equation}
which statisfies $f_{n+2}(i)=0$ for all $1\leq i\leq n+1$.

Therefore by using \eqref{kclasscond} for the edges $(n+2~i)\in\ce_{2n}$ for all $1\leq i\leq n$, one can observe \[f_{n+2}(n+2)\equiv 0~~~\text{mod}~1-e^{2\pi\sqrt{-1}\alpha(n+2~i)}\]for all $1\leq i\leq n$. In particular, from the definition of $\Delta_{\{n+2,\ldots,2n+2\}}$ (see, \eqref{delta_k}), one can write   \[f_{n+2}(n+2)=h_{n+2}\Delta_{\{n+2,\ldots,2n+2\}}(n+2).\]
for some $h_{n+2}\in  R(T^{n+1})\subset \bz[\mathbf{M}, \mathbf{D}]$. We put \[f_{n+3}=f_{n+2}-h_{n+2}\Delta_{\{n+2,\ldots,2n+2\}}\] which satisfies $f_{n+3}(1)=\cdots =f_{n+3}(n+2)=0$ since $\Delta_{\{n+2,\ldots,2n+2\}}(i)=0$ for all $1\leq i\leq n+1$.
Similarly, for  $2\leq j\leq n+2$, there exists $h_{n+j}\in  R(T^{n+1})\subset \bz[\mathbf{M}, \mathbf{D}]$ such that
\begin{equation}
	f_{n+j+1}:=f_{n+j}-h_{n+j}\Delta_{\{n+j,\ldots,2n+2\}}
\end{equation}
 satisfying $f_{n+j+1}(1)=\cdots=f_{n+j+1}(n+j)=0$.

Notice that, when $j=n+2$, \[f_{2n+3}:=f_{2n+2}-h_{2n+2}\Delta_{\{2n+2\}}\] satisfying  $f_{2n+3}(1)=\cdots=f_{{2n+3}}(2n+2)=0$ i.e $f_{2n+3}\equiv 0$.

Therefore, $f_{2n+2}=h_{2n+2}\Delta_{\{2n+2\}}$, and we have
\begin{equation}\label{surj2}
	\begin{split}
		f_{2n+1}&=h_{2n+1}\Delta_{\{2n+1,2n+2\}}+h_{2n+2}\Delta_{\{2n+2\}}\\
		f_{2n}&=h_{2n}\Delta_{\{2n,2n+1,2n+2\}}+h_{2n+1}\Delta_{\{2n+1,2n+2\}}+h_{2n+2}\Delta_{\{2n+2\}}\\
		&~~~~\vdots\\
		f_{n+2}&=h_{n+2}\Delta_{\{n+2,\ldots,2n+2\}}+\cdots+h_{2n}\Delta_{\{2n,2n+1,2n+2\}}+h_{2n+1}\Delta_{\{2n+1,2n+2\}}+h_{2n+2}\Delta_{\{2n+2\}}
\end{split}
\end{equation}
Therefore, we have the following from \eqref{surj1} and \eqref{surj2}, 
\begin{equation}\label{formatkclass}
	\begin{split}
		f=h_1+h_2(1-M_1)&+\cdots+h_{n+1}(1-M_1)\cdots(1-M_n)\\&+h_{n+2}\Delta_{\{n+2,\ldots,2n+2\}}+\cdots+h_{2n+1}\Delta_{\{2n+1,2n+2\}}+h_{2n+2}\Delta_{\{2n+2\}}
	\end{split}
	\end{equation}
	where $h_i\in R(T^{n+1})\subset\bz[\mathbf{M}, \mathbf{D}]$ for each $1\leq i\leq 2n+2$.

Hence, the lemma follows.
\end{proof}

\begin{rema}\label{surjremark1}
	As both ${\bf M}$ and ${\bf D}$ consist of $K$-classes, any $f\in\bz[{\bf M}, {\bf D}]$ is also a $K$-class. Consequently, $f$ can be expressed in the format of \eqref{formatkclass}, by similar arguements as in Lemma \ref{surjmainth}. 
Further,   for any $g\in\bz[\mathscr{Q}_{2n}]=\bz[{\bf M,~D}]/\mathfrak{I}$, one can choose  $f\in\bz[{\bf M}, {\bf D}]$  of the form \eqref{formatkclass} such that $g=f+\mathfrak{I}$.
\end{rema}

Let $v\in \cv_{2n}=[2n+2]$, we define $I_v\subset [n+2]\subset\cv_{2n}$ by 
\[I_v=\begin{cases}
[n+2]\setminus\{v\} &\text{ if } 1\leq v\leq n+1\\
[n+2]\setminus\{\overline{v}\} &\text{ if } n+2\leq v\leq 2n+2\\
\end{cases}\]
\begin{lema}\label{lemma5.4}
	Let $v\in\cv_{2n}$ and $\langle F_J~|~v\notin J\rangle$ be an ideal in $\bz[\mathscr{Q}_{2n}]$ which is generated by $F_J$ (see, \eqref{F_J}) for all $v\notin J$. Then we have the following isomorphism :
	\begin{equation}\label{threeisom}
		\bz[\mathscr{Q}_{2n}]/\langle F_J~|~v\notin J\rangle\simeq \bz[M_i, M_i^{-1}~|~i\in I_v]\simeq R(T^{n+1}).
	\end{equation}
\end{lema}

\begin{proof}
	This is enough to prove the lemma for the vertex $v=1\in \cv_{2n}$. The proof for other vertices will follow similarly.
	
	When $v=1\in \cv_{2n}$, \eqref{threeisom} reduces to the following 
	\[\bz[\mathscr{Q}_{2n}]/\langle F_J~|~1\notin J\rangle\simeq \bz[M_2, M_2^{-1},\ldots,M_{n+2},M_{n+2}^{-1}]\simeq R(T^{n+1}).\]
	
{\bf Claim :} The following canonical morphism (see, \eqref{commdiag}) is surjective :
\[p: \bz[M_2, M_2^{-1},\ldots,M_{n+2},M_{n+2}^{-1}]\to 	\bz[\mathscr{Q}_{2n}]/\langle F_J~|~1\notin J\rangle.\]
Note that the following elements generate $\bz[\mathscr{Q}_{2n}]/\langle F_J~|~1\notin J\rangle$ :
\begin{equation}\label{injproofgen1}
	\{\overline{M_v}, ~~\overline{M_v^{-1}}: v\in\cv_{2n}\}\cup\{\overline{\Delta_P} : P\subset \cv_{2n} ,~ \{i,\overline{i}\}\not\subset P \text{ for all $1\leq i\leq 2n+2$}\},\end{equation} where \[\begin{split}
\overline{M_v}:=M_v+(\mathfrak{I}+\langle F_J~&|~1\notin J\rangle),~~~~~
\overline{M_v^{-1}}:=M_v^{-1}+(\mathfrak{I}+\langle F_J~|~1\notin J\rangle),\\
\overline{\Delta_P}&:=\Delta_P+(\mathfrak{I}+\langle F_J~|~1\notin J\rangle).
\end{split}\]

	Let $L\subset \cv_{2n}$ such that $\{i,\overline{i}\}\not\subset L$ for every $1\leq i\leq n+1$. If $1\in L$  and $|L|<n+1$, then from Lemma \ref{relation4}, we have 
	\[\Delta_L=\Delta_{L\cup\{i\}}\cdot(1-M_i)~~\text {for } i,\overline{i}\notin L\]
	If $1\notin L$, then $\Delta_L=F_L\in \langle F_J~|~1\notin J\rangle$. Therefore, $\overline{\Delta_L}=0$ in $\bz[\mathscr{Q}_{2n}]/\langle F_J~|~1\notin J\rangle$.

	The above implies that every generators $\overline{\Delta_P}\in \bz[\mathscr{Q}_{2n}]/\langle F_J~|~1\notin J\rangle$  can be written by $\overline{\Delta_L}$'s with $1\in L$ and $|L|=n+1$.
	
	We next assume that $I\subset \cv_{2n}$ with $1\notin I$ and $|I|=n$. Further, we assume that there is the unique pair $\{b,\overline{b}\}\not\subset I$. Let $I=\{i_1,\ldots,i_n\}$ and $\cv_{2n}\setminus I=\{1,b,\overline{b}, j_1,\ldots,j_{n-1}\}$. Then by Lemma \ref{relation3}, we have
	\[\Delta_{\{1,\overline{b}, j_1,\ldots,j_{n-1}\}}=(1-M_{i_1})\cdots(1-M_{i_n})-M_b\cdot\Delta_{\{1,b, j_1,\ldots,j_{n-1}\}}\]

	This shows that from the generators $\{\overline{\Delta_L}\}$ with $1\in L$ and $|L|=n+1$, if there is a vertex $\overline{b}\in L$ for some $b=2,\ldots, n+1$, we may replace $\{\overline{\Delta_L}\}$ by $\overline{\Delta_{(L\setminus\{\overline{b}\})\cup \{b\}}}$ by using $\overline{M_b}$.

		Therefore, all other generators of type  $\overline{\Delta_P}$, where $P\subset\cv_{2n}$ and $\{i,\overline{i}\}\not\subset P$ for all $1\leq i\leq {n+1}$, can be replaced by $\overline{\Delta_{\{1,\ldots,n+1\}}}$. 
		Moreover, by Lemma \ref{relation3} for $I=\{n+2,\ldots,2n+1\}\subset[2n+2]$, we have 
\[(1-M_{n+2})\cdots(1-M_{2n+1})=\Delta_{\{1,2,\ldots,n+1\}}+M_1\cdot\Delta_{\{2,\ldots,n+1,2n+2\}}\]
Consequently, since $\Delta_{\{2,\ldots,n+1,2n+2\}}\in \langle F_J~|~1\notin J\rangle$, we have the following  \[(\overline{1-M_{n+2}})\cdots(\overline{1-M_{2n+1}})=\overline{\Delta_{\{1,2,\ldots,n+1\}}}\in \bz[\mathscr{Q}_{2n}]/\langle F_J~|~1\notin J\rangle\]
	This concludes that every element in $\overline{{\bf D}}$ can be written in terms of the elements in $\overline{{\bf M}}$.
	
	Next, from \eqref{F_J}, we have $1-M_1=F_{\cv_{2n}\setminus\{1\}}$ i.e $\overline{M_{1}}=\overline{1}\in\bz[\mathscr{Q}_{2n}]/\langle F_J~|~1\notin J\rangle$.
	Therefore, together with Lemma \ref{relation2}, we have 
	\[\overline{1}\cdot\overline{ M_{\overline{1}}}=\overline{M_2}\cdot\overline{ {M_{\overline{2}}}}=\cdots=\overline{M_{n+1}}\cdot \overline{M_{\overline{n+1}}}=\overline{M_{n+1}}\cdot \overline{M_{n+2}}\in \bz[\mathscr{Q}_{2n}]/\langle F_J~|~1\notin J\rangle.\]
Consequently,  \[\overline{M_{\overline{1}}}=\overline{M_{n+1}}\cdot \overline{M_{n+2}},\text{ and } \overline{M_{\overline{k}}}=\overline{M_{k}^{-1}}\cdot \overline{M_{n+1}}\cdot \overline{M_{n+2}}~~ \text{ for }k=2,\ldots,n\]
 This implies that all other generators in \eqref{injproofgen1} can be reduced to the following
 \[\{\overline{M_i},~\overline{M_i^{-1}} ~|~2\leq i\leq n+2\}\]
	Hence, $p: \bz[M_2, M_2^{-1},\ldots,M_{n+2},M_{n+2}^{-1}]\to 	\bz[\mathscr{Q}_{2n}]/\langle F_J~|~1\notin J\rangle$ is surjective.

	Finally, we consider the homomorphisms
	\[\bz[M_2, M_2^{-1},\ldots,M_{n+2},M_{n+2}^{-1}]\xrightarrow{p}	\bz[\mathscr{Q}_{2n}]/\langle F_J~|~1\notin J\rangle\xrightarrow{\psi}R(T^{n+1})\]
where $\psi$ is induced from the following composition $\bz[\mathscr{Q}_{2n}]\xrightarrow{\phi}K(\mathscr{Q}_{2n})\to R(T^{n+1})$, and is defined by $f\mapsto \phi(f)(1)$ for $f\in\bz[\mathscr{Q}_{2n}]$. Notice that \[\psi\circ p(M_i)=y_{i-1}, ~\text{ and, }~\psi\circ p(M_i^{-1})=y_{i-1}^{-1} ~~\text{ for } i=2,\ldots,n+2.\]

Therefore, $\psi\circ p$ is an isomorphism, hence $p$ is injective. Consequently, $p$ is an isomorphism.

Hence, the lemma follows.
\end{proof}

\begin{coro}
	We have the following injective homomorphism :
	\begin{equation}\label{localization}
		\displaystyle K(\mathscr{Q}_{2n})\hookrightarrow\bigoplus_{v\in\cv_{2n}}R(T^{n+1})\simeq\bigoplus_{v\in\cv_{2n}}	\bz[\mathscr{Q}_{2n}]/\langle F_J~|~v\notin J\rangle\simeq{\bigoplus_{v\in\cv_{2n}}}\bz[M_i, M_i^{-1}~|~i\in I_v].\end{equation}
\end{coro}
\begin{proof}
The first inclusion directly follows from the definition of equivariant $K$-theory of GKM graphs (see, \cite{kr, gsz13}), while the later part follows from Lemma \ref{lemma5.4}.
\end{proof}

\begin{lema}\label{injectivitymainth}
	The homomorphism $\phi: \bz[\mathscr{Q}_{2n}]\to K(\mathscr{Q}_{2n})$ is injective.
\end{lema}
\begin{proof}
This is enough to show that the following  map $\Phi$ is injective :
\[\Phi:\displaystyle  \bz[\mathscr{Q}_{2n}]\xrightarrow{\phi}K(\mathscr{Q}_{2n})\hookrightarrow\bigoplus_{v\in\cv_{2n}}R(T^{n+1})\simeq\bigoplus_{v\in\cv_{2n}}\bz[M_i, M_i^{-1}~|~i\in I_v].\] 

Let $r_u:\displaystyle \bigoplus_{v\in\cv_{2n}}\bz[M_i, M_i^{-1}~|~i\in I_v]\to \bz[M_i, M_i^{-1}~|~i\in I_u]$ be the restriction map at $u\in\cv_{2n}$. For any $f\in \bz[\mathscr{Q}_{2n}]$, let $f(u)$ denote the image of  $f$ by $r_u\circ \Phi$. 

Assume that $\Phi(f)=0$ for an element $f\in \bz[\mathscr{Q}_{2n}]$. Therefore, \[r_v\circ \Phi(f)=f(v)=0 \in \bz[M_i, M_i^{-1}~|~i\in I_v]~~~\text{for all } v\in\cv_{2n}.\]

Note that  $f\in \bz[\mathscr{Q}_{2n}]$ can be written as follows (see,  Remark \ref{surjremark1})
\begin{equation}\label{inj1}
	\begin{split}
		f=h_1+h_2(1-M_1)&+\cdots+h_{n+1}(1-M_1)\cdots(1-M_n)\\&+h_{n+2}\Delta_{\{n+2,\ldots,2n+2\}}+\cdots+h_{2n+1}\Delta_{\{2n+1,2n+2\}}+h_{2n+2}\Delta_{\{2n+2\}}+\mathfrak{I}\\
	\end{split}
\end{equation}
{where $h_i\in R(T^{n+1})\subset\bz[{\bf M}, ~{\bf D}]$ for all $i=1,\ldots, 2n+2$.}

Note that, for all $i=1,\ldots,2n+2$, $\phi(h_i)\in \bz[y_1,y_1^{-1},\ldots,y_{n+1},y_{n+1}^{-1}]$ (see, \eqref{algebramap}). This implies that if for some vertex $u\in\cv_{2n}$, we have $h_i(u)=0$, then $h_i\equiv 0$.

{\bf Claim 1:} $h_i=0$ for all $1\leq i\leq n+1$.\\
Note that 
\begin{equation}\label{inj2}
	\Delta_{\{n+2,\ldots,2n+2\}}(i)=\cdots=\Delta_{\{2n+2\}}(i)=0 ~~\text{for all}~~ i=1,\ldots,n+1,
\end{equation}
and $({1}-M_1)(1)=0$. Therefore, from \eqref{inj1}, we have
\[0=f(1)=h_1(1)+ 0+\cdots+0.\]
 Hence, $h_1=0$.
Further, by using $h_1=0$ in \eqref{inj1} and, for $({1}-M_2)(2)=0$, we have 
\[0=f(2)=h_2({1}-M_1)(2)+ 0+\cdots+0 .\] 
Notice that  $h_2(2), (1-M_1)(2)\in \bz[M_i, M_i^{-1}~|~i\in I_2]$ (see, \eqref{localization}) and $({1}-M_1)(2)\not=0$. Hence $h_2(2)=0$ since $\bz[M_i, M_i^{-1}~|~i\in I_2]$ is an integral domain. Therefore, $h_2=0$.

By following the similar arguements for $i=3,\ldots,n+1$, one can have \[h_3=\cdots=h_{n+1}=0.\]
Therefore,
 \begin{equation}\label{inj3}
	\begin{split}
		f=h_{n+2}\Delta_{\{n+2,\ldots,2n+2\}}+\cdots+h_{2n+1}\Delta_{\{2n+1,2n+2\}}+h_{2n+2}\Delta_{\{2n+2\}}+\mathfrak{I}
	\end{split}
\end{equation}

{\bf Claim 2:} $h_i=0$ for all $n+2\leq i\leq 2n+2$.

From \eqref{inj3}, we have the following for the vertex $n+2\in\cv_{2n}$.
\[	0=f(n+2)=h_{n+2}\Delta_{\{n+2,\ldots,2n+2\}}(n+2)+h_{n+3}\Delta_{\{n+3,\ldots,2n+2\}}(n+2)+\cdots+h_{2n+2}\Delta_{\{2n+2\}}(n+2).\]
Since
\[	\begin{split}
\Delta_{\{n+3,\ldots,2n+2\}}&(n+2)=\cdots=\Delta_{\{2n+2\}}(n+2)=0,\\
	&\Delta_{\{n+2,\ldots,2n+2\}}(n+2)\not=0,    
		\end{split}\]
 we have $h_{n+2}=0$, by the similar reason as above. Iterating the similar arguements for other vertices $n+3,\ldots,2n+2\in\cv_{2n}$, we have  $h_i=0$ for all $n+3\leq i\leq 2n+2$.

Therefore $f=0$, which shows the injectivity of $\Phi$.

Hence, the lemma follows. 
\end{proof}

{\it Proof of Theorem \ref{mainth}. }
The result \eqref{mainth1} follows from Lemma \ref{surjmainth} and Lemma \ref{injectivitymainth}. Further, since $Q_{{{2n}}}$ is an equivariantly formal GKM manifold under the effective $T^{n+1}$-action,  \eqref{mainth2} follows (see \eqref{gkmkth}). \hfill\qedsymbol

{\bf Acknowledgement :} The author is deeply grateful to Prof. Shintar{\^o} Kuroki and Prof. Vikraman Uma for several valuable discussions and for their encouragement during the preparation of the manuscript. The author would also like to thank Prof. Allen Knutson for helpful email exchanges. The author is grateful to the anonymous referees for their careful reading of the manuscript and for their comments and suggestions. The author is also grateful to the Indian Institute of Technology, Madras, for a PhD fellowship.

\end{document}